\documentclass[leqno,12pt]{article} 
\usepackage{amsmath, amssymb}
\usepackage{amsthm} 
\theoremstyle{plain} 
\newtheorem{theorem}{\indent\sc Theorem}[section]

\newtheorem{corollary}[theorem]{\indent\sc Corollary}
\newtheorem{proposition}[theorem]{\indent\sc Proposition}

\theoremstyle{definition} 
\newtheorem{definition}[theorem]{\indent\sc Definition}
\newtheorem{remark}[theorem]{\indent\sc Remark}
\newtheorem{example}[theorem]{\indent\sc Example}

\title{On generalized plastic structures}
\author{Adara M. Blaga and Antonella Nannicini}

\date{}
\begin{document}

\maketitle

\markboth{{\small\it {\hspace{1cm} On generalized plastic structures}}}{\small\it{On generalized plastic structures \hspace{1cm}}}


\begin{abstract}
We introduce the concept of generalized almost plastic structure, and, on a pseudo-Riemannian manifold endowed with two $(1,1)$-tensor fields satisfying some compatibility conditions, we construct a family of generalized almost plastic structures and characterize their integrability with respect to a given affine connection on the manifold.
\end{abstract}


\section{Introduction}
Let $M$ be a smooth manifold and let $E$ be a fiber bundle over $M$. Let us denote by $\Gamma ^\infty (E)$ the set of smooth sections of $E$. A \emph{polynomial structure} on $E$ is a morphism $J:\Gamma ^\infty (E) \rightarrow \Gamma ^\infty (E)$ such that there exists a polynomial $P$ with $P(J)=0$. Almost complex structures, almost product structures, and metallic structures are examples of polynomial structures. A morphism $J$ such that $J^3-pJ-qI=0$, where $p,q$ are positive integers and $I$ is the identity operator on $\Gamma ^\infty (E)$, is called  an \emph{almost nylon structure}, and, if $p=q=1$, then, it is called an \emph{almost plastic structure}. The name plastic goes back to Hans van der Laan (1924) and G. Cordonnier (1928) who discovered architectural proportions based on the unique real solution of the cubic equation $x^3-x-1=0$, given by  $\rho=\sqrt[3]{\frac{9+\sqrt {69}}{18}}+\sqrt[3]{{\frac{9-\sqrt {69}}{18}}}$ and called it, the \emph{plastic number}, \cite{m}.

The present paper aims to focus on almost plastic structures on the generalized tangent bundle $E=TM\oplus T^*M$. In this spirit, we will define a \emph{generalized almost plastic structure} on $M$ as being an almost plastic structure on $E$. Given a pseudo-Riemannian metric on $M$, we construct a family of generalized almost plastic structures induced by two $(1,1)$-tensor fields on $M$ satisfying some compatibility conditions, and we characterize their integrability with respect to a given affine connection on $M$. On the other hand, we define a generalized almost plastic structure by means of a $(1,1)$-tensor field on $M$ satisfying the cubic equation $x^3-x+1=0$ and we establish sufficient integrability conditions in terms of quasi-statistical structures. We remark that this construction gives rise to a {\bf{duality}} between the two equations $x^3-x-1=0$ and $x^3-x+1=0$ in terms of associated generalized structures.

\section{Preliminaries}

\subsection{Plastic matrices}
\begin{definition} Let $\mathbb R (n)$ be the set of real square matrices of order $n$ and let $A\in \nolinebreak \mathbb R (n)$. Then, $A$ is called a \emph{plastic matrix} if $A$ satisfies the equation:
$$A^3-A-I=0,$$
where $I$ is the identity matrix.
\end{definition}

We determine the expressions of the plastic matrices of order $2$, as follows.
\begin{proposition} Let $A \in \mathbb R (2)$ be a plastic matrix. Then, $A$ has one of the following forms:
\begin{equation} \label{M10}A=\rho I,
\end{equation}
where $\rho$ is the plastic number and $I$ is the identity matrix, or
\begin{equation} \label{M20}A=\begin{pmatrix}
                 a_{11} & \displaystyle {\frac {1- a_{11} a_ {22}- a_{11}^2- a_{22}^2}{ a_{21}}} \\
                 a_{21} & a_{22} \\
               \end{pmatrix}
\end{equation}
with $a_{11}, a_{22}\in \mathbb R$, $ a_{21} \in \mathbb R\setminus \{0\}$ and $(trA)^3-trA+1=0$, where $trA:=a_{11}+a_{22}$.
\end{proposition}
\begin{proof} Let
$$A=\begin{pmatrix}
                 a_{11} & a_{12} \\
                 a_{21} & a_{22} \\
               \end{pmatrix}.
$$
We have:
$$A^3=\begin{pmatrix}
                 a_{11}^3+2a_{11} a_{12}a_{21}+ a_{12}a_{21}a_{22} & a_{11}^2 a_{12}+a_{12}^2a_{21}+a_{11}a_{12}a_{22}+a_{12}a_{22}^2 \\
                 a_{11}^2 a_{21}+a_{22}^2a_{21}+a_{11}a_{21}a_{22}+a_{12}a_{21}^2 & a_{22}^3+2a_{22} a_{12}a_{21}+a_{12}a_{21}a_{11}\\
               \end{pmatrix}.
$$
Then, $A$ is a plastic matrix if and only if the following system is satisfied:
$$\left \{ \begin{array}{l}  a_{11}^3+2a_{11} a_{12}a_{21}+ a_{12}a_{21}a_{22}= a_{11}+1\\

\vspace{0.1cm}

a_{12}(a_{11}^2 +a_{12}a_{21}+a_{11}a_{22}+a_{22}^2-1)=0\\

\vspace{0.1cm}

 a_{21}(a_{11}^2 +a_{12}a_{21}+a_{11}a_{22}+a_{22}^2-1)=0\\
\vspace{0.1cm}

a_{22}^3+2a_{22} a_{12}a_{21}+ a_{12}a_{21}a_{11}=a_{22}+1
\end{array}
\right. .$$
If we suppose $a_{21}=0$ (respectively, $a_{12}=0$), then we get that $a_{11}=a_{22}$ is the plastic number and furthermore, also that $a_{12}=0$ (respectively, also that $a_{21}=0$), hence, we have (\ref{M10}).\\
Now, if we suppose $a_{12}a_{21} \neq 0$, then we get the following conditions:
$$\left \{ \begin{array}{l}  a_{11}^3+2a_{11} a_{12}a_{21}+ a_{12}a_{21}a_{22}= a_{11}+1\\

\vspace{0.1cm}

a_{11}^2 +a_{12}a_{21}+a_{11}a_{22}+a_{22}^2=1\\

\vspace{0.1cm}

a_{22}^3+2a_{22} a_{12}a_{21}+ a_{12}a_{21}a_{11}=a_{22}+1
\end{array}
\right. .$$
In particular, we get:
$$\left \{ \begin{array}{l}

a_{12}a_{21}=1-a_{11}a_{22}- a_{11}^2-a_{22}^2\\

\vspace{0.1cm}

(a_{11}+a_{22})^3-(a_{11}+a_{22})+1=0\\

\end{array}
\right. ,$$
hence, we have (\ref{M20}), and the proof is complete.
\end{proof}

\begin{corollary}
Let $A \in \mathbb R (2)$ be a plastic matrix. Then, $A=\rho I$, where $\rho$ is the plastic number and $I$ is the identity matrix, or there exists an invertible matrix $C \in \mathbb R (2)$, such that $A=C^{-1}BC$, where
\begin{equation} \label{M30}B=\begin{pmatrix}
                 \alpha & 1-\alpha^2 \\
                 1 & 0 \\
               \end{pmatrix},
\end{equation}
for $\alpha$ the unique real solution of the equation $x^3-x+1=0$.
\end{corollary}
\begin{proof}
Let $A \in \mathbb R (2)$ be a plastic matrix. If $A$ has a real eigenvalue, then, it has multiplicity two and $A=\rho I$. If $A$ has not real eigenvalues, then, it has two complex conjugate eigenvalues and, as a consequence of the previous proposition, their sum is the real solution, $\alpha$, of the equation $x^3-x+1=0$. In particular, if
\begin{equation*} A=\begin{pmatrix}
                 a_{11} & \displaystyle {\frac {1- a_{11} a_ {22}- a_{11}^2- a_{22}^2}{ a_{21}}} \\
                 a_{21} & a_{22} \\
               \end{pmatrix},
\end{equation*}
then we can choose $$C=\begin{pmatrix}
                 a_{21} & a_ {22}\\
                 0 & 1 \\
               \end{pmatrix}$$
and, as $\alpha=a_{11}+a_{22}$, we get the statement.
\end{proof}

\subsection{Almost plastic structures}

\begin{definition}
Let $M$ be a smooth manifold. An \emph{almost plastic structure} on $M$ is a tensor field $J$, of type $(1,1)$, on $M$, which satisfies the equation:
$$J^3-J-I=0.$$
If $J$ is an almost plastic structure on $M$, then, $(M,J)$ is called an \emph{almost plastic manifold}.
\end{definition}

\begin{definition}
Let $(M,g)$ be a pseudo-Riemannian manifold. An \emph{almost plastic pseudo-Riemannian structure} on $(M,g)$ is an almost plastic structure $J$ on $M$ such that $g(JX,Y)=g(X,JY)$ for all $X,Y\in\Gamma^\infty(TM)$. If $J$ is an almost plastic pseudo-Riemannian structure on $(M,g)$, then, $(M,g,J)$ is called an \emph{almost plastic pseudo-Rieman\-ni\-an manifold}.
\end{definition}

Let $N(J)$ be the Nijenhuis tensor field of $J$, defined as:
$$N(J)(X,Y):=[JX,JY]-J[JX,Y]-J[X,JY]+J^2[X,Y]$$
for all $X,Y\in\Gamma^\infty(TM)$.

According to \cite{v}, we say that an almost plastic structure $J$ is \emph{integrable} if $N(J)=0$. We will call \emph{plastic structure} an integrable almost plastic structure. An almost plastic manifold $(M,J)$ with $J$ integrable will be called a \emph{plastic manifold}.

\begin{example} Let $\rho=\sqrt[3]{\frac{9+\sqrt {69}}{18}}+\sqrt[3]{{\frac{9-\sqrt {69}}{18}}}$ be the plastic number. Then, $J=\rho I$ is a pseudo-Riemannian plastic structure on any pseudo-Riemannian manifold $(M,g)$ which is called the \emph {trivial pseudo-Riemannian plastic structure}.
\end{example}

\begin{remark}
We remark that the concept of "almost plastic Riemannian structure" has no sense because, excepting the trivial case $J=\rho I$, a plastic structure is not real diagonalizable and then it cannot be symmetric with respect to a positive definite scalar product.
\end{remark}

\begin{remark}
We remark that an almost plastic structure $J$ is invertible and the inverse tensor field is given by $J^{-1}=J^2-I$. Indeed, $J(J^2-I)=J^3-J=I$. However, $J^{-1}$ is not an almost plastic structure because $(J^{-1})^3-J^{-1}-I=-J$.
\end{remark}

\begin{remark}
A \emph{metallic structure} on $M$, \cite{metalic}, is a tensor field $J$, of type $(1,1)$, on $M$, which satisfies the equation:
$$J^2-pJ-qI=0,$$
where $p,q$ are positive integers. We remark that, if $J$ is a metallic structure, then, $J$ is an almost plastic structure if and only if
$$\left \{ \begin{array}{l}

p^3-p+1=0\\

\vspace{0.1cm}

q=1-p^2\\

\end{array}
\right. \ \ \textrm{or} \ \ \left \{ \begin{array}{l}

q\neq 1-p^2\\

\vspace{0.1cm}

J=\displaystyle{\frac{1-pq}{p^2+q-1}}I\\

\end{array}
\right. ,$$
respectively, $J$ satisfies a \textbf{dual} equation, $J^3-J+I=0$, if and only if
$$\left \{ \begin{array}{l}

p=\rho \ \ \textrm{is the plastic number}\\

\vspace{0.1cm}

q=1-\rho^2\\

\end{array}
\right. \ \ \textrm{or} \ \ \left \{ \begin{array}{l}

q\neq 1-p^2\\

\vspace{0.1cm}

J=-\displaystyle{\frac{1+pq}{p^2+q-1}}I\\

\end{array}
\right. .$$
\end{remark}

\section{Plastic structures in generalized geometry}

\subsection{Geometrical properties of the generalized tangent bundle}

Let $TM\oplus T^*M$ be the generalized tangent bundle of $M$. On $TM\oplus T^*M$, we consider the natural indefinite metric
\begin{equation*}
<X+\eta,Y+\beta>:=-\frac{1}{2}(\eta(Y)+\beta(X))
\end{equation*}
and the natural symplectic structure
\begin{equation*}
(X+\eta,Y+\beta):=-\frac{1}{2}(\eta(Y)-\beta(X))
\end{equation*}
for all $X,Y\in \Gamma^{\infty}(TM)$ and $\eta,\beta\in \Gamma^\infty(T^*M).$

If $g$ is a non-degenerate $(0,2)$-tensor field on $M$, we will denote by $g$ the flat isomorphism $\flat_g:\Gamma^{\infty}(TM)\rightarrow \Gamma^{\infty}(T^*M)$, $\flat_g(X):=i_Xg$, and by $g^{-1}$ its inverse, and we define the bilinear form, $\check g$, on $TM\oplus T^*M$ by:
\begin{equation*}
\check g(X+\eta, Y+\beta):=g(X,Y)+g(g^{-1}(\eta),g^{-1}(\beta))
\end{equation*}
for all $X,Y\in \Gamma^\infty(TM)$ and $\eta,\beta\in \Gamma^\infty(T^*M)$.

Furthermore, given an affine connection $\nabla$ on $M$, we define the affine connections $\hat{\nabla}$ and $\check{\nabla}$ on $TM \oplus T^*M$ by:
\begin{equation*}
\hat{\nabla}_{X+\eta}(Y+\beta):=\nabla_XY+g(\nabla_X(g^{-1}(\beta)))
\end{equation*}
and
\begin{equation*}
\check{\nabla}_{X+\eta}(Y+\beta):=\nabla_XY+\nabla_X\beta
\end{equation*}
for all $X,Y\in \Gamma^\infty(TM)$ and $\eta,\beta\in \Gamma^\infty(T^*M)$.
We remark that $\hat \nabla =\check \nabla$ if and only if $\nabla g=0$.

Finally, we define the bracket $[\cdot,\cdot]_{\nabla}$:
\begin{equation*}
[X+\eta,Y+\beta]_{\nabla}:=[X,Y]+\nabla_X\beta-\nabla_Y\eta
\end{equation*}
for all $X,Y\in \Gamma^\infty(TM)$ and $\eta,\beta\in \Gamma^\infty(T^*M)$.

\subsection{Generalized almost plastic structures}

Any morphism $\hat J:\Gamma^{\infty}(TM\oplus T^*M)\rightarrow \Gamma^{\infty}(TM\oplus T^*M)$ will be called
a \textit{generalized structure on $M$}.

\begin{definition} Let $\hat J:\Gamma^{\infty}(TM\oplus T^*M)\rightarrow \Gamma^{\infty}(TM\oplus T^*M)$ be a generalized structure on $M$. If $\hat J$ satisfies the equation
$${\hat J}^3-\hat J-I=0,$$
then, $\hat J$ is called a \emph{generalized almost plastic structure on $M$}.
\end{definition}

\begin {example} Let $M$ be a smooth manifold and let $J_1,J_2$ be almost plastic structures on $M$. Then, the generalized structure
\begin{equation} \label{m10}\hat{J}:=\begin{pmatrix}
                 J_1 & 0 \\
                 0 & J_2^* \
               \end{pmatrix},
\end{equation}
where $J_2^*:\Gamma^ \infty(T^*M)\rightarrow \Gamma^ \infty (T^*M)$ is defined by $J_2^*(\eta)(X):=\eta(J_2X)$ for all $\eta \in \Gamma^ \infty(T^*M)$ and $X\in \Gamma ^\infty (TM)$, is a generalized almost plastic structure on $M$. For $J_1=J_2=J$, we will denote the induced generalized almost plastic structure  by:
\begin{equation} \label{M100}\hat{J}:=\begin{pmatrix}
                 J & 0 \\
                 0 & J^* \\
               \end{pmatrix}.
\end{equation}
\end{example}

\begin{proposition} Let $\hat J$ be the generalized almost plastic structure defined by (\ref{M100}). Then,
\begin{equation*}
<\hat J(X+\eta),Y+\beta>=<X+\eta,\hat J(Y+\beta)>;
\end{equation*}
moreover, if $g$ is a pseudo-Riemannian metric and $(M,g,J_1), (M,g,J_2)$ are almost plastic pseudo-Riemannian manifolds, then, the structure ${\hat J}$ defined by (\ref{m10}) satisfies:
\begin{equation*}
\check g(\hat J(X+\eta), Y+\beta)=\check g(X+\eta,\hat J(Y+\beta))
\end{equation*}
for all $X,Y\in \Gamma^\infty(TM)$ and $\eta,\beta\in \Gamma^\infty(T^*M)$.
\end{proposition}
\begin{proof}
We have: $$<\hat J(X+\eta), Y+\beta>=<JX+J^*(\eta),Y+\beta>=-\frac{1}{2}(\eta(JY)+\beta(JX));$$
on the other hand, $$<X+\eta, \hat J(Y+\beta)>=<X+\eta,JY+J^*(\beta)>=-\frac{1}{2}(\eta(JY)+\beta(JX)),$$
and the first statement is proved.
Moreover, ${\hat J}$ defined by (\ref{m10}) satisfies:
\begin{align*}
\check g({\hat J}(X+\eta), Y+\beta)&=\check g(J_1X+J_2^*(\eta),Y+\beta)=g(J_1X,Y)+g(g^{-1}(J_2^*(\eta)),g^{-1}(\beta))\\
&=g(X,J_1Y)+g(J_2(g^{-1}(\eta)),g^{-1}(\beta))\\
&=g(X,J_1Y)+g(g^{-1}(\eta),J_2(g^{-1}(\beta)))\\
&=g(X,J_1Y)+g(g^{-1}(\eta),g^{-1}(J_2^*(\beta)))\\
&=\check g(X+\eta,{\hat J}(Y+\beta)),
\end{align*}
and the proof is complete.
\end{proof}

\begin{proposition}
Let $\hat J$ be the generalized almost plastic structure defined by (\ref {m10}). Then, $\hat \nabla \hat J=0$ if and only if
$$\left \{ \begin{array}{l}

\nabla J_1=0\\

\vspace{0.1cm}

\nabla J_2=0\\
\end{array}
\right. .$$
Moreover, $\check \nabla  \hat J=0$ if and only if
$$\left \{ \begin{array}{l}

\nabla J_1=0\\

\vspace{0.1cm}

\nabla J_2=0\\
\end{array}
\right. .$$
\end{proposition}
\begin{proof}
A direct computation gives:
$$(\hat \nabla \hat J)_{X+\eta}(Y+\beta)=(\nabla_X J_1)Y+g((\nabla _X J_2)(g^{-1}(\beta)))$$
for all $X, Y\in \Gamma^\infty (TM)$ and $\eta, \beta\in \Gamma^\infty (T^*M)$, and the first statement is proved.
Moreover, a direct computation gives:
$$(\check \nabla \hat J)_{X+\eta}(Y+\beta)=(\nabla_X J_1)Y+(\nabla_XJ_2^*)\beta$$
for all $X, Y\in \Gamma^\infty (TM)$ and $\eta, \beta\in \Gamma^\infty (T^*M)$, and the proof is complete.
\end{proof}

\subsection{Generalized almost plastic structures defined by two $(1,1)$-tensor fields}

Taking inspiration from the form of order two plastic matrices, we now construct a family of generalized almost plastic structures on a pseudo-Riemannian manifold.

\begin{proposition}
Let $(M,g)$ be a pseudo-Riemannian manifold and let $J_1$, $J_2$ be two $(1,1)$-tensor fields on $M$ satisfying the following conditions:\\
(i) $ J_1J_2=J_2J_1$\\
(ii) $(J_1+J_2)^3-(J_1+J_2)+I=0$\\
(iii) $g(J_1X,Y)=g(X,J_1Y)$, for all $X,Y\in \Gamma^\infty(TM)$\\
(iv) $g(J_2X,Y)=g(X,J_2Y)$, for all $X,Y\in \Gamma^\infty(TM)$. Then,
\begin{equation} \label{m15}\hat{J}:=\begin{pmatrix}
                 J_1 & (I-J_1J_2-J^2_1-J^2_2)g^{-1} \\
                 g & J_2^* \\
               \end{pmatrix}
\end{equation}
is a generalized almost plastic structure on $M$.
\end{proposition}
\begin{proof}
From
$$\hat J(X+\eta)=\left(J_1X+(I-J_1J_2-J^2_1-J^2_2)g^{-1}(\eta)\right)+\left(g(X)+J^*_2(\eta)\right),$$ by using the properties (i), (iii), (iv), we get:
\begin{align*}
\hat J^2(X+\eta)&=\left(J_1^2X+(I-J_1J_2-J^2_1-J^2_2)J_1g^{-1}(\eta)+ (I-J_1J_2-J^2_1-J^2_2)(X+J_2g^{-1}(\eta))\right)\\
&\hspace{12pt}+\left(g(J_1X)+ (I-J_1J_2-J^2_1-J^2_2)^*(\eta)+g(J_2X)+(J_2^*)^2(\eta)\right)\\
&=\left((I-J_1J_2-J^2_2)X+(J_1+J_2-2J_1^2J_2-2J^2_2J_1-J_1^3-J_2^3)g^{-1}(\eta)\right)\\
&\hspace{12pt}+\left(g(J_1X)+ (I-J_1J_2-J^2_1)^*(\eta)+g(J_2X)\right);
\end{align*}
moreover,
\begin{align*}
\hat J^3(X+\eta)&=\left((J_1+ (J_1+J_2)-(J_1+J_2)^3)X+(I-J_1J_2-J_1^2-J_2^2)g^{-1}(\eta)\right)\\
&\hspace{12pt}+\left(g(X)+(J_2^*+(J_1^*+J_2^*)-(J_1^*+J_2^*)^3(\eta)\right),
\end{align*}
and the proof is complete.
\end{proof}

\begin{remark} From the proof of the previous proposition, we get that, if $(J_1+J_2)$ is an almost plastic structure on $M$, then, the generalized structure $\hat J$ defined by (\ref{m15}) satisfies the condition $\hat J^3-\hat J+I=0$. Therefore, this construction of a generalized structure defines a {\bf duality} between the two equations $x^3-x-1=0$ and $x^3-x+1=0$.
\end{remark}

\begin{proposition}
Let $\hat J$ be the generalized almost plastic structure defined by (\ref {m15}). Then, $\hat \nabla \hat J=0$ if and only if $$\left \{ \begin{array}{l}

\nabla J_1=0\\

\vspace{0.1cm}

\nabla J_2=0\\
\end{array}
\right. .$$
Moreover, $\check \nabla  \hat J=0$ if and only if $$\left \{ \begin{array}{l}

\nabla J_1=0\\

\vspace{0.1cm}

\nabla J_2=0\\

\vspace{0.1cm}

\nabla g=0 \\
\end{array}
\right. .$$
\end{proposition}
\begin{proof} A direct computation gives:
$$(\hat \nabla \hat J)_{X+\eta}(Y+\beta)=(\nabla_X J_1)Y+(\nabla _X (I-J_1J_2-J_1^2-J_2^2))(g^{-1}(\beta))+g((\nabla_X J_2)(g^{-1}(\beta)))$$
for all $X, Y\in \Gamma^\infty (TM)$ and $\eta, \beta\in \Gamma^\infty (T^*M)$, and the first statement is proved.
Moreover, a direct computation gives:
$$(\check \nabla \hat J)_{X+\eta}(Y+\beta)=(\nabla_X J_1)Y+(\nabla _X (I-J_1J_2-J_1^2-J_2^2)g^{-1})(\beta)+(\nabla_X g)Y+(\nabla_XJ_2^*)\beta$$
for all $X, Y\in \Gamma^\infty (TM)$ and $\eta, \beta\in \Gamma^\infty (T^*M)$, and the proof is complete.
\end{proof}

\section{Integrability of some generalized almost plastic structures}

Let $M$ be a smooth manifold and let $\nabla$ be an affine connection on $M$.
\begin{definition}
A generalized structure $\hat{J}$ on $M$ is called {\it $\nabla$-integrable} if its Nijenhuis tensor field $N^{\nabla}({\hat{J}})$ with respect to $\nabla$:
$$N^{\nabla}({\hat{J}})(\sigma, \tau):=[\hat{J}\sigma,\hat{J}\tau]_{\nabla}-\hat{J}[\hat{J}\sigma, \tau]_{\nabla}-\hat{J}[\sigma, \hat{J}\tau]_{\nabla} +\hat{J}^{2}[\sigma, \tau]_{\nabla}$$ vanishes for all $\sigma,\tau \in {\Gamma}^{\infty}(TM\oplus T^*M)$.
\end{definition}

We shall study the integrability of some generalized almost plastic structures, characterizing it also in terms of quasi-statistical structures.
We recall the following definition.

\begin{definition} \cite{k, ma}
Let $g$ be a non-degenerate $(0,2)$-tensor field on $M$ and let $\nabla$ be an affine connection on $M$ with torsion operator $T^{\nabla}$.
Then, $(g,\nabla)$ is called a \textit{quasi-statistical structure} on $M$
if
$$(\nabla _X g)(Y,Z)-(\nabla _Y g)(X,Z)+g(T^{\nabla}(X,Y),Z)=0$$
for all $X, Y, Z \in \Gamma^{\infty} (TM)$, where $T^{\nabla}$ is the torsion of the given connection $\nabla$.
If $(g,\nabla)$ is a quasi-statistical structure on $M$, then, $(M,g,\nabla)$ is called a \textit{quasi-statistical manifold}.
\end{definition}
For generalized quasi-statistical structures, see \cite{124, 96,94}.

\subsection{Integrability of generalized almost plastic structures induced by two almost plastic structures on $M$}

\begin {proposition} Let $M$ be a smooth manifold, let $\nabla$ be an affine connection on $M$, and let $J_1,J_2$ be two almost plastic structures on $M$. Then, the generalized almost plastic structure defined by
$$\hat{J}:=\begin{pmatrix}
                 J_1 & 0 \\
                 0 & J_2^* \
               \end{pmatrix},
$$
is $\nabla$-integrable if and only if the following conditions are satisfied:
\begin{equation} \label{M15}\left \{ \begin{array}{l} N(J_1) = 0\\

\vspace{0.1cm}

{\nabla}_{J_1X}J_2= J_2({\nabla}_{X}J_2)
\end{array}
\right.
\end{equation}
for all $X\in \Gamma^\infty (TM)$.
\end{proposition}
\begin{proof} A direct computation gives:
\begin{align*}
N^\nabla(\hat J)(X+\eta,Y+\beta)&=N(J_1)(X,Y)\\
&\hspace{12pt}+((\nabla_{J_1X}J_2^*)-J_2^*(\nabla_{X}J_2^*))(\beta)-((\nabla_{J_1Y}J_2^*)-J_2^*(\nabla_{Y}J_2^*))(\eta)
\end{align*}
for all $X,Y\in \Gamma^\infty (TM)$ and $\eta, \beta \in \Gamma^\infty (T^*M)$.
Moreover, for any $X,Y\in \Gamma^\infty (TM)$ and for any $\eta \in \Gamma^\infty (T^*M)$:
$$((\nabla_{J_1X}J_2^*)(\eta))(Y)=\eta((\nabla_{J_1X}J_2)Y)$$
and
$$((J_2^*(\nabla_{X}J_2^*))(\eta))(Y)=\eta(J_2(\nabla_{X}J_2)Y),$$
and the proof is complete.
\end{proof}

\begin{remark}
We remark that if $J_1=J_2=J$, then (\ref{M15}) becomes:
\begin{equation*}\left \{ \begin{array}{l} T^{\nabla}(JX,JY)-JT^{\nabla}(JX,Y)-JT^{\nabla}(X,JY)+J^2T^{\nabla}(X,Y)=0\\

\vspace{0.1cm}

{\nabla}_{JX}J= J({\nabla}_{X}J)
\end{array}
\right.
\end{equation*}
for all $X,Y\in \Gamma^\infty (TM)$.
In particular, for a torsion-free affine connection, the integrability condition becomes:
$${\nabla}_{JX}J= J({\nabla}_{X}J)$$
for all $X\in \Gamma^\infty (TM)$.
\end{remark}

\subsection{Integrability of generalized almost plastic structures induced by a polynomial structure $J$ on $M$ such that $J^3-J+I=0$}

Taking inspiration from (\ref {M30}), we consider the generalized almost plastic structure defined by:
\begin{equation} \label{M45}\hat{J}:=\begin{pmatrix}
                 J & (I-J^2)g^{-1} \\
                 g &0 \\
               \end{pmatrix},
\end{equation}
where $J$ is a polynomial structure on $M$ such that $J^3-J+I=0$.

\begin{proposition} Let $g$ be a pseudo-Riemannian metric and let $\nabla$ be an affine connection on $M$. If $J$ is integrable, $\nabla J=0$, and $(M,g,\nabla)$ is a quasi-statistical manifold, then $\hat J$ given by (\ref{M45}) is $\nabla$-integrable.
\end{proposition}
\begin{proof}
A direct computation gives:
\begin{align*}
N^\nabla(\hat J)(X,Y)&=N(J)(X,Y)+(I-J^2)g^{-1}((\nabla_Yg)X-(\nabla_Xg)Y+g(T^\nabla(Y,X)))\\
&\hspace{12pt}+(\nabla_{JX}g)Y-(\nabla_Yg)JX+g(T^\nabla(JX,Y))\\
&\hspace{12pt}+(\nabla_Xg)JY-(\nabla_{JY}g)X+g(T^\nabla(X,JY))\\
&\hspace{12pt}+(\nabla_Yg)X-(\nabla_Xg)Y+g(T^\nabla(Y,X))+(\nabla_Y J^*)g(X)-(\nabla_X J^*)g(Y)\\
&\hspace{12pt}+J^*((\nabla_Yg)X-(\nabla_Xg)Y+g(T^\nabla(Y,X)))
\end{align*}
for all $X,Y\in\Gamma^\infty(TM)$;
\begin{align*}
N^\nabla(\hat J)(X,g(Z))&=-(\nabla _{JX}J^2)Z-(\nabla_{(I-J^2)Z }J)X+J(\nabla_XJ^2)Z\\
&\hspace{12pt}-T^\nabla(JX, (I-J^2)Z)+JT^\nabla(X,(I-J^2)Z)\\
&\hspace{12pt}-(I-J^2)g^{-1}((\nabla_{JX}g)Z)+J(I-J^2)(g^{-1}(\nabla_Xg)Z)\\
&\hspace{12pt}+(\nabla_{X}g)Z-(\nabla_Zg)X+g(T^\nabla(X,Z))\\
&\hspace{12pt}-(J^*)^2((\nabla_Xg)Z)+(\nabla_{J^2Z}g)X-g(T^\nabla(X,J^2Z))\\
&=-g^{-1}((\nabla_{JX}g)(I-J^2)Z))-(\nabla_{(I-J^2)Z} g)JX+g(T^\nabla(JX,(I-J^2)Z)))\\
&\hspace{12pt}+Jg^{-1}((\nabla_Xg)(I-J^2)Z)-(\nabla_{(I-J^2)Z} g)X+g(T^\nabla(X,(I-J^2)Z)))\\
&\hspace{12pt}+(\nabla_{X}g)Z-(\nabla_Zg)X+g(T^\nabla(X,Z))\\
&\hspace{12pt}-(\nabla_Xg)J^2Z+(\nabla_{J^2Z}g)X-g(T^\nabla(X,J^2Z))
\end{align*}
for all $X,Z\in\Gamma^\infty(TM)$;
\begin{align*}
N^\nabla(\hat J)(g(Z),g(W))&=g^{-1}((\nabla_Wg)Z-(\nabla_Zg)W+g(T^\nabla(W,Z)))\\
&\hspace{12pt}+g^{-1}((\nabla_{J^2W}g)J^2Z-(\nabla_{J^2Z}g)J^2W+g(T^\nabla(J^2W,J^2Z)))\\
&\hspace{12pt}+g^{-1}((\nabla_Zg)J^2W-(\nabla_{J^2W}g)Z+g(T^\nabla(Z,J^2W)))\\
&\hspace{12pt}+g^{-1}((\nabla_{J^2Z}g)W-(\nabla_Wg)J^2W+g(T^\nabla(J^2Z,W)))
\end{align*}
for all $Z,W\in\Gamma^\infty(TM)$, and we get the conclusion.
\end{proof}

\noindent Adara M. BLAGA, \\
Department of Mathematics, \\
Faculty of Mathematics and Computer Science, \\
West University of Timi\c{s}oara, \\
Bld. V. P\^{a}rvan 4, 300223, Timi\c{s}oara, Romania, \\
Email: adarablaga@yahoo.com

\bigskip

\noindent Antonella NANNICINI, \\
Department of Mathematics and Informatics "U. Dini", \\
University of Florence,\\
Viale Morgagni 67/a, 50134, Firenze, Italy,\\
Email: antonella.nannicini@unifi.it

\end{document}